\documentclass{amsart}
\usepackage{amssymb,amsmath,amsfonts,amsthm,mathrsfs}
\usepackage[colorlinks=true, citecolor=blue, anchorcolor=red]{hyperref}
\usepackage{enumerate}
\usepackage{bbm}
\usepackage{tikz-cd}
\makeatletter
\@namedef{subjclassname@2020}{%
  \textup{2020} Mathematics Subject Classification}
\makeatother

\newtheorem{theorem}{Theorem}[section]

\theoremstyle{definition}

\newtheorem{example}[theorem]{Example}

\theoremstyle{remark}
\newtheorem{remark}[theorem]{Remark}

\numberwithin{equation}{section}

\renewcommand{\epsilon}{\varepsilon}

\newcommand{\N}{\mathbb{N}}
\newcommand{\Z}{\mathbb{Z}}
\newcommand{\R}{\mathbb{R}}
\newcommand{\C}{\mathbb{C}}

\newcommand{\txte}{\text{e}}
\newcommand{\txtd}{\text{d}}

\newcommand{\dist}{\operatorname{dist}}

\begin{document}

\title[Slow manifolds in infinite dimensions]{Connecting a direct and a Galerkin approach to slow manifolds in infinite dimensions}

\author{Maximilian Engel}
\address{Department of Mathematics and Computer Science, Freie Universit\"{a}t Berlin, 14195 Berlin, Germany}
\email{maximilian.engel@fu-berlin.de}

\thanks{ME was supported by Germany's Excellence Strategy -- The Berlin Mathematics Research Center MATH+ (EXC-2046/1, project ID: 390685689).}

\author{Felix Hummel}
\address{Faculty of Mathematics, Technical University 
of Munich, 85748 Garching b.~M\"unchen, Germany}
\email{hummel@ma.tum.de}
\thanks{FH acknowledges partial support via the SFB/TR109 ``Discretization in Geometry and Dynamics'' as well as partial support of the EU within the TiPES project funded the European Unions Horizon 2020 research and innovation programme un der grant agreement No. 820970.}

\author{Christian Kuehn}
\address{Faculty of Mathematics, Technical University 
of Munich, 85748 Garching b.~M\"unchen, Germany}
\email{ckuehn@ma.tum.de}
\thanks{CK acknowledges support via a Lichtenberg Professorship as well as support via the SFB/TR109 ``Discretization in Geometry and Dynamics'' as well as partial support of the EU within the TiPES project funded the European Unions Horizon 2020 research and innovation programme un der grant agreement No. 820970.}

\thanks{This project is TiPES contribution {\#}80: This project has received funding from the European Union’s Horizon 2020 research and innovation programme under grant agreement No 820970}


\date{February 10, 2021}



\begin{abstract}
In this paper, we study slow manifolds for infinite-dimensional evolution equations. We compare two approaches: an abstract evolution equation framework and a finite-dimensional spectral Galerkin approximation. We prove that the slow manifolds constructed within each approach are asymptotically close under suitable conditions. The proof is based upon Lyapunov-Perron methods and a comparison of the local graphs for the slow manifolds in scales of Banach spaces. In summary, our main result allows us to change between different characterizations of slow invariant manifolds, depending upon the technical challenges posed by particular fast-slow systems.  
\end{abstract}

\maketitle

\section{Introduction}

The perturbation theory of normally hyperbolic invariant manifolds introduced by Fenichel \cite{Fenichel1,WigginsIM} has proved to be a useful tool in the theory of dynamical systems. One important consequence of Fenichel's works is that they provide a suitable framework for the treatment of fast-slow systems~\cite{Fenichel4,Jones} of the form
\begin{align}\label{Eq:FastSlow}
	\begin{aligned}
		\epsilon \partial_t u^{\epsilon} &= A u^{\epsilon} + f(u^{\epsilon},v^{\epsilon}),\\
				\partial_t v^{\epsilon} &= B v^{\epsilon} + g(u^{\epsilon},v^{\epsilon}),
	\end{aligned}
\end{align}
where $0\leq\epsilon\ll1$ is a small parameter, $A,B$ are matrices, and $f,g$ are differentiable nonlinearities. The unknown functions $u^{\epsilon}$ and $v^{\epsilon}$ are called fast and slow variable, respectively. System~\eqref{Eq:FastSlow} is already written in a variant of (local) Fenichel normal form~\cite{Fenichel4,Jones} separating matrices $A,B$ and the nonlinearities $f,g$, which is also a quite natural form in the PDE context to be considered below. For the classical finite-dimensional case, Fenichel's techniques are also known as geometric singular perturbation theory. The main result is that -- under suitable assumptions -- for all $\epsilon>0$ small enough there is a manifold $S_{\epsilon}$ which is locally invariant under the flow generated by \eqref{Eq:FastSlow} and which can be written as a graph over the slow variable. More precisely, one may write
\[
	S_{\epsilon}:=\{(h^{\epsilon}(v),v): v\in Y\},
\]
where $X$ and $Y$ are the finite-dimensional vector spaces $u^{\epsilon}$ and $v^{\epsilon}$, respectively, take values in, and $h^{\epsilon}\colon Y\to X$ is a Lipschitz continuous function. These manifolds, which are called slow manifolds, are $\epsilon$-close over compact subsets in $Y$ to the critical manifold
\[
	S_{0}:=\{(h^{0}(v),v): v\in Y\},
\]
where $h^0(v)$ denotes the unique solution of 
\[
0 = A h^0(v) + f(h^0(v),v).
\]
Moreover, the flow on $S_{\epsilon}$ converges to the slow flow on $S_0$ which is defined to be the flow which is generated by the singular limit of \eqref{Eq:FastSlow} as $\epsilon\to 0$, that is
\begin{align}\label{Eq:FastSlow=0}
	\begin{aligned}
		0 &= A u^{0} + f(u^{0},v^{0}),\\
				\partial_t v^{0} &= B v^{0} + g(u^{0},v^{0}).
	\end{aligned}
\end{align}
The existence of such slow manifolds is usually taken as a formal justification for the intuitive idea, that after a short initial time the dynamics of \eqref{Eq:FastSlow} only evolve on the slow time scale and are described well by the slow subsystem \eqref{Eq:FastSlow=0}. Since 
\[
0 = A u + f(u,v)
\]
is supposed to have the unique solution $u=h^0(v)$, one may rewrite \eqref{Eq:FastSlow=0} as
\begin{align}\label{Eq:FastSlow_reduced}
	\partial_t v^{0} &= B v^{0} + g(h^0(v^{0}),v^{0}),\quad u^0=h^0(v^0).
\end{align}
Altogether, we can then reduce \eqref{Eq:FastSlow} to \eqref{Eq:FastSlow_reduced}. The advantage of \eqref{Eq:FastSlow_reduced} is that the fast variable is now uniquely determined by the slow variable, i.e., the dimension of the dynamical problem \eqref{Eq:FastSlow} has been reduced.\\
It has been an open problem for a few decades, how to generalize Fenichel theory to the infinite-dimensional setting, with fast-slow systems of partial differential equations as an important application. Even though persistence of normally hyperbolic invariant manifolds in Banach spaces was derived by Bates, Lu and Zeng in \cite{Bates_Lu_Zeng_1998} for bounded semiflow perturbations, the existence of slow manifolds for PDEs, involving spatial differential operators in the slow variable equations, had only been shown in very special cases such as for the Maxwell-Bloch equations~\cite{Menon_Haller_2001}. Recently, there have been two new attempts to provide techniques for a geometric singular perturbation theory in infinite dimensions: In \cite{Engel_Kuehn_2020}, slow manifolds in infinite dimensions were approximated by finite-dimensional slow manifolds within a Galerkin procedure, paving the way for an extension of geometric blow-up from ODEs to PDEs. A more direct approach was taken in \cite{Hummel_Kuehn_2020}, where a two-parameter family $S_{\epsilon,\zeta}$ of slow manifolds was contructed via a Lyapunov-Perron argument. The main ingredient of the latter procedure is a splitting of the slow variable space $Y=Y_{F}^{\zeta}\oplus Y_{S}^{\zeta}$ into a quickly decaying part and a part on which the linear dynamics are invertible. 
We will introduce both approaches in Section~\ref{sec:two_approaches} and provide a precise comparison result in Section~\ref{Sec:MainResult}, relating the two types of slow manifolds to each other via estimates for their distance and its decay in $\epsilon, \zeta$. Finally, in Section~\ref{sec:case_study}, we exemplify this main result at the hand of a slow-fast PDE with fast reaction-diffusion dynamics, also discussing intricacies of the Galerkin limits.

\section{The two approaches}\label{Sec:TwoApproaches}
\label{sec:two_approaches}

\subsection{Assumptions}
In the following, we discuss in detail the assumptions for the subsequent statements. It is, in fact, one of the main difficulties in infinite-dimensional geometric singular perturbation theory to find conditions, which allow for the construction of slow manifolds and are, at the same time, satisfied in many important applications. Although the list of assumptions we impose is quite long, it has already been demonstrated in \cite{Hummel_Kuehn_2020} that the conditions are satisfied for a large class of PDEs, e.g.~reaction-diffusion systems; in comparison to \cite{Hummel_Kuehn_2020}, we add a few assumptions which allow us to trade regularity for better estimates. Moreover, we also add a splitting in the fast variable space so that we can define an appropriate Galerkin approximation. 

In the following, let $n\in\N$.

\subsubsection{Assumption $(A_n)$}\label{Assump:A_n}
We consider the fast-slow system \eqref{Eq:FastSlow} on Banach spaces $X$ and $Y$, supplemented by the initial conditions
\begin{align}\label{Eq:FastSlow:InitialValues}
	u^{\epsilon}(0)=u_0 \in X_n,\quad v^{\epsilon}(0)=v_0 \in Y_n,
\end{align}
where $X_n, Y_n$ are elements of the interpolation-extrapolation scales introduced hereafter (see also Appendix~\ref{sec:int_ext_scales}) and we have $0=Au_0+f(u_0,v_0)$ for $\epsilon =0$. Assume further that the nonlinearities satisfy $f(0,0)=0$ and $g(0,0)=0$.
Then the following conditions ensure that \eqref{Eq:FastSlow} together with \eqref{Eq:FastSlow:InitialValues} has a unique solution $(u^{\epsilon},v^{\epsilon})\in C^{1}([0,\infty);X_{n-1}\times Y_{n-1})\cap C([0,\infty);X_{n}\times Y_{n})$ which is approximated well by the slow flow in a sense which we will make precise later. 
\begin{enumerate}[(i)]
	\item \textbf{Generation of semigroups}: the closed linear operator $A\colon X\supset D(A)\to X$ generates an exponentially stable $C_0$-semigroup $(\txte^{tA})_{t\geq0}\subset\mathcal{B}(X)$ on the Banach space $X$. The closed linear operator $B\colon Y\supset D(B)\to Y$ is the generator of a $C_0$-semigroup $(\txte^{tB})_{t\geq0}\subset\mathcal{B}(Y)$ on the Banach space $Y$.
	 \item \textbf{Generation of Banach scales}: the interpolation-extrapolation scales generated by $(X,A)$ and $(Y,B)$ (see Appendix~\ref{sec:int_ext_scales}) are --- up to uniform equivalence of norms for each fixed $\alpha_0\in[-1,\infty)$ and all $\alpha\in[-1,\alpha_0]$ --- given by $(X_{\alpha})_{\alpha\in[-1,\infty)}$ and $(Y_{\alpha})_{\alpha\in[-1,\infty)}$. If $0\notin\rho(B)$, then $(Y_{\alpha})_{\alpha\in[-1,\infty)}$ shall be equivalent to the interpolation-extrapolation scale generated by $B-\lambda$ for some $\lambda\in\rho(B)$.
	 \item \textbf{Bounded Fr\'{e}chet derivatives}: let $\gamma_X\in(0,1]$ if $(\txte^{tA})_{t\geq0}\subset\mathcal{B}(X)$ is holomorphic and $\gamma_X=1$ otherwise. In addition, we choose $\delta_X\in[1-\gamma_X,1]$. Let further $\delta_Y\in(0,1]$ if $(\txte^{tB})_{t\geq0}\subset\mathcal{B}(Y)$ is holomorphic and $\delta_Y=1$ otherwise. The nonlinearities $f\colon X_{n-1+\delta_X}\times Y_{n-\delta_X}\to X_{n-1}$ and $g\colon X_{n}\times Y_{n}\to Y_{n-1+\delta_Y}$ are continuously differentiable and there are constants $L_f,L_g>0$ (which may depend on $n$) such that      
     \begin{align*}
         \| Df(x,y)\|_{\mathcal{B}(X_n\times Y_n, X_{n-1+\gamma_X})}&<L_f\quad(x\in X_n,y\in Y_n),\\
         \| Df(x,y)\|_{\mathcal{B}(X_{n-1+\delta_X}\times Y_{n-1}, X_{n-2+\delta_X})}&<L_f\quad(x\in X_{X_{n-1+\delta_X}},y\in Y_{n-1}),\\
         \| Dg(x,y)\|_{\mathcal{B}(X_n\times Y_n, Y_{n-1+\delta_Y})}&<L_g\quad(x\in X_n,y\in Y_n).
     \end{align*}
    \item \textbf{Bounds for semigroups}: we choose constants $M_A,M_B,C_A,C_B>0$ (which may depend on $n$) as well as $\omega_A<0$ and $\omega_B\in\R$ (which do not depend on $n$) such that for all $t>0$
 \begin{align*}
 	\|\txte^{tA}\|_{\mathcal{B}(X_n)}\leq &M_A \txte^{\omega_A t},\quad \|\txte^{tA}\|_{\mathcal{B}(X_{n-1+\gamma_X},X_n)}\leq C_At^{\gamma_X-1} \txte^{\omega_A t},\\
 	& \|\txte^{tA}\|_{\mathcal{B}(X_{n-1+\delta_X},X_n)}\leq C_At^{\delta_X-1} \txte^{\omega_A t}
 \end{align*}
 and
 \begin{align*}
  	\|\txte^{tB}\|_{\mathcal{B}(Y_n)}\leq M_B \txte^{\omega_B t},\quad \|\txte^{tB}\|_{\mathcal{B}(Y_{n-1+\delta_Y},Y_n)}\leq C_Bt^{\delta_Y-1} \txte^{\omega_B t}.
 \end{align*}
   \item \textbf{Relation of constants}: we define $\omega_f:=\omega_A+(2C_AL_{f})^{\frac{1}{\gamma_X}} (\frac{1}{\gamma_X})^{\frac{1-\gamma_X}{\gamma_X}}$ if $\gamma_X\in(0,1)$ and take $\omega_f>\omega_A+C_AL_f$ if $\gamma_X=1$. Moreover, we assume
 \begin{align*}
 	&\omega_f<0,\\L_f\max\{\|A^{-1}\|_{\mathcal{B}(X_{\gamma_X},X_{1})}&,\|A^{-1}\|_{\mathcal{B}(X_{\delta_X-1},X_{\delta_X})}\}<1,
 \end{align*}
\end{enumerate}
\begin{remark}
	The conditions of Assumption $(A_n)$ are almost identical to the ones in \cite[Section 4]{Hummel_Kuehn_2020}. Here they are slightly simplified in the sense that the differentiability of the nonlinearities, which is assumed here, is not necessary for all results in \cite{Hummel_Kuehn_2020}.
\end{remark}

\subsubsection{Assumption $(B_n)$} \label{Assump:B_n}
This assumption is sufficient for obtaining a two-parameter family of slow manifolds $S_{\epsilon,\zeta}$ \cite{Hummel_Kuehn_2020}, in particular specifying the role of the second parameter $\zeta$:
we assume that for each small $\zeta>0$ there is a splitting $Y=Y_F^{\zeta}\oplus Y_S^{\zeta}$, independently from $n$, into a \emph{fast} part $Y_F^{\zeta}$ and a \emph{slow} part $Y_S^{\zeta}$ such that the projections $\operatorname{pr}_{Y_{F}^{\zeta}}$ and $\operatorname{pr}_{Y_{S}^{\zeta}}$ commute with $B$ on $Y_n$. 

The crucial characterization of the fast part is that $Y_F^{\zeta}\cap Y_{n-1+\delta_Y}$ contains the parts of $Y_{n-1+\delta_Y}$ that decay under the semigroup $(\txte^{tB})_{t\geq0}$ almost as fast as the space $X_n$ under $(\txte^{\zeta^{-1}tA})_{t\geq0}$; analogously, the slow space $Y_S^{\zeta}\cap Y_n$ contains the parts of $Y_n$ which do not decay or which only decay slowly under the semigroup $(\txte^{tB})_{t\geq0}$ compared to $X_n$ under $(\txte^{\zeta^{-1}tA})_{t\geq0}$. This idea is expressed in point~\eqref{Assump:B_n:5} of the following assumptions:
\begin{enumerate}[(i)]
	\item \label{Assump:B_n:1} \textbf{Closed subspaces}: the spaces $Y_{F}^{\zeta}\cap Y_{\beta}$ and $Y_{S}^{\zeta}\cap Y_{\beta}$ are closed in $Y_{\beta}$ for all $\beta\geq0$ and will be endowed with the norms $\|\cdot\|_{Y_\beta}$.
	   \item \textbf{Lipschitz bound}: using the notation $f(x,y_F,y_S):=f(x,y_F+y_S)$ and $g(x,y_F,y_S):=g(x,y_F+y_S)$, the nonlinearity $g$ satisfies
        \begin{align*}
        & \|\operatorname{pr}_{Y_S^{\zeta}}g(x-\tilde{x},y_F-\tilde{y_F},y_S-\tilde{y_S})\|_{Y_n} \\
        & \leq L_g\zeta^{\delta_Y-1}\big(\|x-\tilde{x}\|_{X_n}+\|y_F-\tilde{y_F}\|_{Y_n}+\|y_S-\tilde{y_S}\|_{Y_n}\big).
        \end{align*}
       \item \textbf{Semigroup in slow subspace}: the realization of $B$ in $Y_S^{\zeta}\cap Y_{n-1}$, i.e.
        \[
              B_{Y_S^{\zeta}\cap Y_{n-1}}\colon Y_S^{\zeta}\cap Y_{n-1}\supset D(B_{Y_S^{\epsilon}\cap Y_{n-1}})\to Y_S^{\zeta}\cap Y_{n-1},\; v\mapsto Bv
         \]
         with 
         \[
         	D(B_{Y_S^{\zeta}\cap Y_{n-1}}):=\{v_0\in Y_S^{\zeta} \cap Y_{n}:Bv_0\in Y_S^{\zeta}\cap Y_{n-1} \}
         \]
         generates a $C_0$-group $(\txte^{tB_{Y_S^{\zeta}\cap Y_{n-1}}})_{t\in\R}\subset\mathcal{B}((Y_S^{\zeta}\cap Y_{n-1},\|\cdot\|_{Y_{n-1}}))$ which satisfies $\txte^{tB_{Y_S^{\zeta}\cap Y_{n-1}}}=\txte^{tB}$ on $Y_S^{\zeta}\cap Y_{n-1}$ for $t\geq0$. For the sake of readability, we will still write $B$ instead of $B_{Y_S^{\zeta}\cap Y_{n-1}}$.
         \item \textbf{Semigroup in fast subspace}: the realization of $B$ in $Y_F^{\zeta}\cap Y_{n-1}$, i.e.
        \[
          B_{Y_F^{\zeta}\cap Y_{n-1}}\colon Y_F^{\zeta}\cap Y_{n-1}\supset D(B_{Y_F^{\zeta}\cap Y_n})\to Y_F^{\zeta}\cap Y_{n-1},\; v\mapsto Bv
        \]
        with 
         \[
        	D(B_{Y_F^{\zeta}}):=\{v_0\in Y_F^{\zeta}\cap Y_n:Bv_0\in Y_F^{\zeta}\cap Y_{n-1} \}
        \]
        has $0$ in its resolvent set. For the sake of readability, we will still write $B$ instead of $B_{Y_F^{\zeta}\cap Y_{n-1}}$.
        \item \label{Assump:B_n:5} \textbf{Speed of decay in $Y_F^{\zeta}$ and $Y_S^{\zeta}$}:
        there are constants $C_B,M_B>0$ such that for all $\zeta>0$ small enough there are constants $0\leq N_F^{\zeta}< N_S^{\zeta}<-\zeta^{-1}\omega_A$ such that for all $t>0$, $y_F\in Y_F^{\zeta}\cap Y_{n-1+\delta_Y}$ and $y_S\in Y_S^{\zeta}\cap Y_n$ we have the estimates
        \begin{align*}
         \|\txte^{tB}y_F\|_{Y_n}&\leq  C_Bt^{\delta_Y-1}\txte^{(N_F^{\zeta}+\zeta^{-1}\omega_A)t} \|y_F\|_{Y_{n-1+\delta_Y}},\\
         \|\txte^{-tB}y_S\|_{Y_n}&\leq M_B \txte^{-(N_S^{\zeta}+\zeta^{-1}\omega_A) t} \|y_S\|_{Y_n}.
        \end{align*}
           \item \textbf{Estimate for contraction property in Lyapunov-Perron argument}: the parameters and constants introduced above satisfy
            \begin{align} \label{Eq:Estimate_for_FixedPoint}
            	\frac{2^{\gamma_X}L_fC_A\Gamma(\gamma_X)}{\big(2(\epsilon\zeta^{-1}-1)\omega_A+\epsilon(N_S^{\zeta}+N_F^{\zeta})\big)^{\gamma_X}}+\frac{2^{\delta_Y}L_gC_B\Gamma(\delta_Y)}{(N_S^{\zeta}-N_F^{\zeta})^{\delta_Y}}+\frac{2\zeta^{\delta_Y-1}L_gM_B}{N_S^{\zeta}-N_F^{\zeta}}<1,
             \end{align}
             where $\Gamma$ denotes the gamma function.
\end{enumerate}
\begin{remark}
	Assumption $(B_n)$ is identical to the conditions in \cite[Section 5]{Hummel_Kuehn_2020} except for the fact that in \cite[Section 5]{Hummel_Kuehn_2020} it is only assumed for $n=1$. Here, we make use of additional regularity in certain estimates and therefore formulate the assumption for $n\in\N$.
\end{remark}

\subsubsection{Assumption $(C_n)$}
	If we want to use a Galerkin approximation in both the slow and the fast variable, then it is useful to also impose similar conditions on $X$, i.e. that there is a splitting $X=X^{\zeta}_F\oplus X^{\zeta}_S$ such that the conditions \eqref{Assump:B_n:1}-\eqref{Assump:B_n:5} in Assumption $(B_n)$ hold with $Y$ and $B$ being replaced by $X$ and $A$, respectively.
	
\subsubsection{Assumption $(D)$}
	This assumption will enable us to trade regularity for additional decay behavior. We assume that for $0\leq \alpha \leq \beta$ there is a constant $C_{\alpha,\beta}$ such that, for all $x\in X_F^{\zeta}\cap X_\beta,\,y\in Y_F^{\zeta}\cap Y_\beta$, we have the estimates
\begin{align}\label{Eq:Smoothness_for_decay}
	\| y \|_{Y_\alpha} \leq C_{\alpha,\beta} \zeta^{\beta-\alpha}\| y \|_{Y_\beta},\quad \| x \|_{X_\alpha} \leq C_{\alpha,\beta} \zeta^{\beta-\alpha}\| x \|_{X_\beta}.
\end{align}

\begin{remark}
		Let us give an example of a situation in which Assumption $(D)$ is satisfied. We define the Bessel potential space on the torus $\mathbb{T}$ by
	\[
		H^s(\mathbb{T}):=\left\{u\in\mathscr{D}'(\mathbb{T}): \sum_{l\in\Z} (1+|l|^2)^{s/2}\langle u,e_l \rangle_{L_2(\mathbb{T})} e_l \, \in L_2(\mathbb{T}) \right\},
	\]
	where $e_l=[x\mapsto e^{2\pi ilx}]$, and endow the space with the norm
	\[
		\| u \|_{H^s(\mathbb{T})}:= \| ((1+|l|^2)^{s/2}\langle u,e_l \rangle_{L_2(\mathbb{T})})_{l\in\Z}\|_{\ell_2(\Z)}.
	\]
	Consider for example $X=Y=L_2(\mathbb{T})$ and $A=B=\Delta-1$ with domain $D(A)=D(B)=Y_1=X_1=H^2(\mathbb{T})$. The interpolation-extrapolation scales are then given by $Y_{\alpha}=X_{\alpha}=H^{2\alpha}(\mathbb{T})$. $Y_F^{\zeta}$ and $X_F^{\zeta}$ will be the subspaces of $L_2(\mathbb{T})$ such that the $l$-th Fourier coefficients with $(|l|-1)^2\leq \zeta^{-1}$ are equal to $0$. With this choice we obtain for $y \in Y_F^{\zeta} \cap Y_{\beta}$
	\begin{align*}
		&\quad\| y \|_{Y_\alpha}=\| y \|_{H^{2\alpha}(\mathbb{T})}=\left(\sum_{l\in\Z,\;|l|^2\geq\zeta^{-1}}(1+|l|^2)^{2\alpha}\langle y,e_l\rangle_{L_2(\mathbb{T})}^2\right)^{1/2}\\
		&\leq (1+\zeta^{-1})^{\alpha-\beta} \left(\sum_{l\in\Z,\;|l|^2\geq\zeta^{-1}}(1+|l|^2)^{2\beta}\langle y,e_l\rangle_{L_2(\mathbb{T})}^2\right)^{1/2}\lesssim \zeta^{\beta-\alpha}\| y \|_{Y_\beta},
	\end{align*}
	and the same for $x \in X_F^{\zeta} \cap X_{\beta}$.
	This is estimate \eqref{Eq:Smoothness_for_decay} with $C_{\alpha,\beta}=1$.
\end{remark}

\subsection{The direct approach} Let us now briefly collect the main results of \cite[Section 5]{Hummel_Kuehn_2020}. Under the assumptions $(A_n)$ and $(B_n)$, one can rewrite \eqref{Eq:FastSlow} together with \eqref{Eq:FastSlow:InitialValues} as 
\begin{align}
\begin{aligned}\label{Eq:Fast-Slow_Splitting}
 \epsilon\partial_t u^{\epsilon}(t) &= Au^{\epsilon}(t)+ f(u^\epsilon(t),v^{\epsilon}_F(t),v^{\epsilon}_S(t)),\\
 \partial_t v^{\epsilon}_F(t) &= Bv_F^{\epsilon}(t)+\operatorname{pr}_{Y_F^{\zeta}}g(u^\epsilon(t),v^{\epsilon}_F(t),v^{\epsilon}_S(t)),\\
  \partial_t v^{\epsilon}_S(t) &= Bv_S^{\epsilon}(t)+\operatorname{pr}_{Y_S^{\zeta}}g(u^\epsilon(t),v^{\epsilon}_F(t),v^{\epsilon}_S(t)),\\
 u^{\epsilon}(0)&=u_0,\quad v^{\epsilon}_F(0)=\operatorname{pr}_{Y_F^{\zeta}}v_0,\quad v^{\epsilon}_S(0)=\operatorname{pr}_{Y_S^{\zeta}}v_0.
 \end{aligned}
\end{align}
For this equation, we can formulate the following theorem, which is a collection of the results in \cite[Section 5]{Hummel_Kuehn_2020}.
\begin{theorem}
    Let $n\in\N$ and suppose that Assumption $(A_n)$ and Assumption $(B_n)$ hold true. Fix $c\in(0,1)$ and let $0<\epsilon<c\frac{\omega_f}{\omega_A}\zeta$. Then there is a family of sets $S_{\epsilon,\zeta}$ given as graphs
    \[
        S_{\epsilon,\zeta}:=\{(h_X^{\epsilon,\zeta}(v_{0,S}),h_{Y_F^{\zeta}}^{\epsilon,\zeta}(v_{0,S}),v_{0,S})\,\vert\, v_{0,S}\in Y_S^{\zeta}\cap Y_n\}
    \]
    where $(h_X^{\epsilon,\zeta},h_{Y_F^{\zeta}}^{\epsilon,\zeta})\colon Y_S^{\zeta}\cap Y_n\to X_n\times (Y_F^{\zeta}\cap Y_n)$ is differentiable such that the following assertions hold:
    \begin{enumerate}[(a)]
        \item \textbf{Invariance}: the set $S_{\epsilon,\zeta}$ is invariant under the semiflow generated by \eqref{Eq:Fast-Slow_Splitting}.
        \item \textbf{Distance between $S_{\epsilon,\zeta}$ and critical manifold}: there is a constant $C>0$, which is independent of $\epsilon$ and $\zeta$, such that for all $v_{0,S}\in Y_S^{\zeta}\cap Y_n$ we have
        	\[
        		\left\| \begin{pmatrix} h^{\epsilon,\zeta}_{X_n}(v_{0,S})-h^0(v_{0,S}) \\
	        	 h^{\epsilon,\zeta}_{Y^{\zeta}_F}(v_{0,S})\end{pmatrix} \right\|_{X_n\times Y_n}\leq C\left(\epsilon       +\frac{1}{(N_S^{\zeta}-N_F^{\zeta})^{\delta_y}}\right)\|v_{0,S}\|_{Y_n}.
        	\]
        \item \textbf{Exponentially fast convergence to $S_{\epsilon,\zeta}$}: there are $\epsilon_0,\zeta_0>0$ and constants $C,c>0$ independent of $T$ such that for all $\epsilon\in(0,\epsilon_0]$, $\zeta\in(0,\zeta_0]$ with $0<\epsilon<c\frac{\omega_f}{\omega_A}\zeta$, all $t\in[0,T]$ and all $v_{0,S}\in Y_S^{\zeta}\cap Y_n$ we have
        \begin{align*}
		\left\|\begin{pmatrix} u^{\epsilon}(t)-h^{\epsilon,\zeta}_{X_n}(v^{\epsilon}_S(t))\\
						v^{\epsilon}_F(t)-h^{\epsilon,\zeta}_{Y_F^{\zeta}}(v^{\epsilon}_S(t))\end{pmatrix}\right\|_{X_n\times Y_n}\leq C\txte^{-ct}\left\|\begin{pmatrix} u_0-h^{\epsilon,\zeta}_{X_n}(v_{0,S})\\
						v_{0,F}-h^{\epsilon,\zeta}_{Y_F^{\zeta}}(v_{0,S})\end{pmatrix}\right\|_{X_n\times Y_n}.
	\end{align*}
	\item \textbf{Approximation by slow subsystem}: the reduced slow subsystem given by
    	\begin{align}
        	\begin{aligned}\label{Eq:Reduced_Slow_Subsystem}
        		0&=Au^{0}_{\zeta}(t)+f(u^{0}_{\zeta}(t),v^0_{\zeta}(t)),\\
        		0&=\operatorname{pr}_{Y_F^{\zeta}}v^0_{\zeta}(t),\\
        		\partial_tv^0_{\zeta}(t)&=Bv^0_{\zeta}(t)+\operatorname{pr}_{Y_S^{\zeta}}g(u^{0}_{\zeta}(t),v^0_{\zeta}(t)),\\
        		v^0_{\zeta}(0)&=\operatorname{pr}_{Y_S^{\zeta}}v_0.
        	\end{aligned}
        \end{align}
        has a unique solution $(u^{0}_{\zeta},v^{0}_{\zeta})\in C^1([0,T];X_{n-1}\times Y_{n-1})\cap C([0,T];X_n\times Y_n)$ which approximates the solution of the full fast-slow system. 
        
        More precisely, there are a constant $C>0$ which may depend on $T$, some suitably chosen $\omega_g \in \R$ and $\epsilon_0,\zeta_0>0$ such that for all $\epsilon\in(0,\epsilon_0]$, $\zeta\in(0,\zeta_0]$ with $0<\epsilon<c\frac{\omega_f}{\omega_A}\zeta$, all $t\in[0,T]$ and all $v_{0}\in  Y_n$ we have
        \begin{align*}
		\left\|\begin{pmatrix}u^{\epsilon}(t)-h^0(v^{0}_{\zeta}(t))\\ v^{\epsilon}(t)-v^{0}_{\zeta}(t) \end{pmatrix}\right\|_{Y_n}
		&\leq C\bigg(\|\operatorname{pr}_{Y_F^{\zeta}}v_0\|_{Y_n}+\big(\epsilon+\tfrac{1}{(\omega_g-\zeta^{-1}\omega_A-N_F^{\zeta})^{\delta_Y}}\big)\|v_0\|_{Y_n}\\
		&\qquad\qquad+(\epsilon^{\delta_Y}+\txte^{\epsilon^{-1}\omega_f t})\|u_0-h^0(v_0)\|_{X_n}\bigg).
	\end{align*}
	If even $(u_0,\operatorname{pr}_{Y_F^{\zeta}}v_0,\operatorname{pr}_{Y_S^{\zeta}}v_0)\in S_{\epsilon,\zeta}$, then it holds that
		\begin{align*}
		\left\|\begin{pmatrix}u^{\epsilon}(t)-h^0(v^{0}_{\zeta}(t))\\ v^{\epsilon}(t)-v^{0}_{\zeta}(t) \end{pmatrix}\right\|_{X_n\times Y_n}\leq C\left(\epsilon+\tfrac{1}{(\omega_g-\zeta^{-1}\omega_A-N_F^{\zeta})^{\delta_Y}}+\tfrac{1}{(N_S^{\zeta}-N_F^{\zeta})^{\delta_Y}}\right)\|v_0\|_{Y_n}.
	\end{align*}
    \end{enumerate}
\end{theorem}

\subsection{The Galerkin approach}\label{Sec:Galerkin}
Assuming conditions $(A_n)$, $(B_n)$ and $(C_n)$, we may additionally consider the projection of \eqref{Eq:FastSlow} and \eqref{Eq:FastSlow:InitialValues} to the slow part in both variables, i.e.
\begin{align}
	\begin{aligned}\label{Eq:Fast-Slow_Galerkin}
		\epsilon\partial_t u^{\epsilon}_G&= Au^{\epsilon}_G +\operatorname{pr}_{X_S^\zeta}f(u^{\epsilon}_G,v_G^{\epsilon}),\\
		\partial_t v^{\epsilon}_G &= Bv^{\epsilon}_G + \operatorname{pr}_{Y_S^{\zeta}}g(u^{\epsilon}_G,v_G^{\epsilon}),\\
		u_G^{\epsilon}(0)&=\operatorname{pr}_{X_S^{\zeta}}u_0,\quad v_G^{\epsilon}(0)=\operatorname{pr}_{Y_S^{\zeta}}u_0.
	\end{aligned}
\end{align}
Note that, in general, this is not necessarily a finite-dimensional evolution equation. If $A$ and $B$ generate $C_0$-groups, one may even have $X_S^{\zeta}=X$ and $Y_S^{\zeta}=Y$. However, if $A$ and $B$ have eigenvalue expansions with only a finite number of eigenvalues in each vertical stripe of bounded width within the complex plane, the spaces $X_S^{\zeta}$ and $Y_S^{\zeta}$ are finite-dimensional. This is the situation of many applications such as the Laplacian $\Delta$ on the torus $\mathbb{T}$. In such a case \eqref{Eq:Fast-Slow_Galerkin} indeed coincides with a Galerkin approximation of \eqref{Eq:FastSlow} and \eqref{Eq:FastSlow:InitialValues}. The spaces $X_S^{\zeta}$ and $Y_S^{\zeta}$ are given as the linear span of the eigenfunctions associated with the $N_A$ and $N_B$ eigenvalues, including multiplicities, of $A$ and $B$, respectively, in $\{z\in\C: \zeta^{-1}\omega_A+N_S^{\zeta}\leq\Re(z)\}$.

For example, for $A=B=\Delta$ on $X=Y=L_2(\mathbb{T})$ with eigenvalues $\lambda_k = -4\pi^2 k^2$, $k \in \mathbb{Z}$, and eigenfunctions $(e_k)_{k \in \mathbb{Z}}$, the expansions
\begin{equation*}
u^{\varepsilon}(x,t)=\sum_{k\in \mathbb{Z}} u_k^{\varepsilon}(t) e_k(x),\qquad v^{\varepsilon}(x,t)=\sum_{k\in \mathbb{Z}} v_k^{\varepsilon}(t) e_k(x)
\end{equation*}
give, upon taking the inner product of~\eqref{Eq:FastSlow} with each $e_k$, the system of Galerkin ODEs
\begin{align}
\begin{aligned}\label{Eq:Fast-Slow_Galerkin_ODES}
\varepsilon \partial_t u_k^{\varepsilon}&=\lambda_k u_k^{\varepsilon} + \langle f(u^{\varepsilon}, v^{\varepsilon}),e_k\rangle \\
\partial_t v_k^{\varepsilon} &=  \lambda_k v_k^{\varepsilon} +\langle g(u^{\varepsilon}, v^{\varepsilon}),e_k\rangle. 
\end{aligned}
\end{align}
A truncation at $\left|k\right| \leq k_0 = N_{\Delta}$, for $\epsilon$ and $\zeta$ given as before, yields the described correspondence to system~\eqref{Eq:Fast-Slow_Galerkin}. Such a Galerkin approach is very insightful in situations of dynamical interest such as dynamic bifurcations in reaction-diffusion systems where geometric techniques can be applied to the finite-dimensional approximation and then be extended to the infinite-dimensional limit (see e.g.~\cite{Engel_Kuehn_2020}).

Generally, the existence of such an infinite-dimensional limit raises questions:
if we fix $\zeta$ and let $\epsilon>0$ be small enough, then \eqref{Eq:Fast-Slow_Galerkin} has a family of slow manifolds $G_{\epsilon,\zeta}$ given by
\[
G_{\epsilon,\zeta}= \{(h_G^{\epsilon,\zeta}(v),v):v\in Y_S^{\zeta} \cap Y_n \}
\]
for certain mappings $h_G^{\epsilon,\zeta}\colon Y_S^{\zeta} \cap Y_n \to X_S^{\zeta} \cap X_n$. If $Y_S^{\zeta} \cap Y_n$ and $X_S^{\zeta} \cap X_n$ are finite-dimensional, then this can be derived by classical finite-dimensional Fenichel theory. If they are infinite-dimensional, then one can still use the results from \cite{Hummel_Kuehn_2020}. 
It is now of particular interest to study the behavior of $G_{\epsilon,\zeta}$ -- or a similar object related to it -- as $\zeta\to0$, which corresponds with $N_A, N_B \to \infty$ in the situation described above.
Under a suitable notion of convergence, the potential limiting object $G_{\epsilon,0}$ may be considered as a type of slow manifold. The main difficulty for such an approach is that the existence of $G_{\epsilon,\zeta}$ for fixed $\epsilon$ becomes unclear when $\zeta$ gets too small. In \cite{Engel_Kuehn_2020}, such a procedure was carried out for a particular example, by using an explicit approximation of the parametrization $h_G^{\epsilon,\zeta}$ whose limit for $\zeta \to 0$ could be obtained directly.
Generally, one has to be careful about changing the order of quantifiers for $\epsilon$ and $\zeta$ and the dynamical interpretation of the different objects, as we will demonstrate in Example~\ref{Ex:explicit_calculations}.
Hence, it becomes particularly important to understand the relation between $G_{\epsilon, \zeta}$ and $S_{\epsilon, \zeta}$ in order to measure the quality of the Galerkin approximation for infinite-dimensional fast-slow systems.

\section{The main result}\label{Sec:MainResult}
In this section, we fix $m,n\in\N$, $m\leq n$. For our main result we suppose that --- as it was derived in \cite{Hummel_Kuehn_2020} --- the slow manifolds have been constructed via a Lyapunov-Perron approach. Therefore, we have that
\begin{align*}
	\begin{pmatrix} h_X^{\epsilon,\zeta}(v_{0,S}) \\ h_{Y_F^{\zeta}}^{\epsilon,\zeta}(v_{0,S}) \end{pmatrix} = \begin{pmatrix} \int_{-\infty}^0 e^{-\epsilon^{-1}sA} f(\bar{u}(s),\bar{v}_F(s),\bar{v}_S(s))\,\txtd s \\  \int_{-\infty}^0 e^{-sB} \operatorname{pr}_{Y_F^{\zeta}} g(\bar{u}(s),\bar{v}_F(s),\bar{v}_S(s))\,\txtd s\end{pmatrix},
\end{align*}
where $v_{0,S}\in Y_S^{\zeta}\cap Y_n$ and where $(\bar{u},\bar{v}_F,\bar{v}_S)$ denotes the unique fixed point of the operator
\begin{align}
\begin{aligned}\label{Eq:LyapunovPerron:1}
	&\mathscr{L}_{v_{0,S},\epsilon,\zeta}\colon \mathcal C_{\eta,n}\to \mathcal C_{\eta,n},\\
	&\quad\begin{pmatrix}u\\v_F\\v_S \end{pmatrix}\mapsto \left[t\mapsto\begin{pmatrix} 
		\epsilon^{-1} \int_{-\infty}^t \txte^{\epsilon^{-1}(t-s)A} f(u(s),v_F(s),v_S(s))\,\txtd s \\
		\int_{-\infty}^t \txte^{(t-s)B} \operatorname{pr}_{Y_F^{\zeta}}g(u(s),v_F(s),v_S(s))\,\txtd s\\
		\txte^{tB}v_{0,S}+\int_{0}^t \txte^{(t-s)B} \operatorname{pr}_{Y_S^{\zeta}}g(u(s),v_F(s),v_S(s))\,\txtd s
	\end{pmatrix}\right].
	\end{aligned}
\end{align}
Here, the space $\mathcal C_{\eta, n}$ with
$$\eta:=\zeta^{-1}\omega_A+\frac{N_S^{\zeta}+N_F^{\zeta}}{2}$$
consists of all $(u,v_F,v_S)\in C((-\infty,0];X\times (Y_F^{\zeta}\cap Y_n)\times (Y_S^{\zeta}\cap Y_n))$ such that
	\[
	\|(u,v_F,v_S)\|_{\mathcal C_{\eta,n}}:=\sup_{t\leq0 } \txte^{-\eta t} \big(\|u(t)\|_{X_n}+\|v_F(t)\|_{Y_n}+\|v_S(t)\|_{Y_n}\big)<\infty.
\]
The fixed point $(\bar{u},\bar{v}_F,\bar{v}_S)$ has been shown to exist in \cite[Proposition 5.1]{Hummel_Kuehn_2020}. Likewise, we may write
\begin{align}
	h_G^{\eta,\zeta}(v_{0,S})=\epsilon^{-1}\int_{-\infty}^0 e^{-\epsilon^{-1}sA}\operatorname{pr}_{X_S^{\zeta}}f(\bar{u}_G(s),\bar{v}_G(s))\,\txtd s,
\end{align}
where $(\bar{u}_G,\bar{v}_G)$ denotes the unique fixed point of the operator
\begin{align}
\begin{aligned}\label{Eq:LyapunovPerron:2}
	&\mathscr{L}^G_{v_{0,S},\epsilon,\zeta}\colon \mathcal C_{\eta,n}^G\to \mathcal C_{\eta,n}^G,\\
	&\quad\begin{pmatrix}u_G\\v_G \end{pmatrix}\mapsto \left[t\mapsto\begin{pmatrix} 
		\epsilon^{-1} \int_{-\infty}^t \txte^{\epsilon^{-1}(t-s)A} \operatorname{pr}_{X_S^{\zeta}}f(u_G(s),v_G(s))\,\txtd s \\
		\txte^{tB}v_{0,S}+\int_{0}^t \txte^{(t-s)B} \operatorname{pr}_{Y_S^{\zeta}}g(u_G(s),v_G(s))\,\txtd s
	\end{pmatrix}\right].
	\end{aligned}
\end{align}
Analogously to before, $\mathcal C_{\eta, n}^G$ denotes the space of all $(u_G,v_G)\in C((-\infty,0];(X_S^{\zeta}\cap X_n) \times (Y_S^{\zeta}\cap Y_n))$ such that
	\[
	\|(u_G,v_G)\|_{\mathcal C_{\eta,n}^G}:=\sup_{t\leq0 } \txte^{-\eta t} \big(\|u_G(t)\|_{X_n}+\|v_G(t)\|_{Y_n}\big)<\infty.
\]

With this terminology at hand, we can now formulate our main theorem:
\begin{theorem}\label{Thm:Main}
	We fix $m,n\in\N$, $m\leq n$ and $c\in(0,1)$. Suppose that the assumptions $(A_m)$, $(B_m)$ and $(C_m)$ as well as $(A_n)$, $(B_n)$, $(C_n)$ and $(D)$ are satisfied. Then there is a constant $C>0$ such that for all $\epsilon,\zeta>0$ small enough with $c\frac{\omega_f}{\omega_A}\zeta>\epsilon$ and all $v_{0,S}\in Y_S^{\zeta}\cap Y_n$, we have
		\begin{align}\label{Eq:Comparison_slow_manifolds}
		\begin{aligned}
		&\|h^{\epsilon,\zeta}_X(v_{0,S})-h^{\epsilon,\zeta}_G(v_{0,S})\|_{X_m} + \|h^{\epsilon,\zeta}_{Y_F^{\zeta}}(v_{0,S})\|_{Y_m}\\
		&\leq C\left(\frac{\zeta^{n-m}}{(N_S^{\zeta}-N_F^{\zeta})^{\delta_Y}}+\zeta^{n-m+\gamma_X}\right)\|v_{0,S}\|_{Y_n}.
		\end{aligned}
	\end{align}
\end{theorem}

\begin{proof}
	As above, let $(\bar{u},\bar{v}_F,\bar{v}_S)$ be the unique fixed point of $\mathscr{L}_{v_{0,S},\epsilon,\zeta}$ from \eqref{Eq:LyapunovPerron:1} and $(\bar{u}_G,\bar{v}_G)$ the one of $\mathscr{L}_{v_{0,S},\epsilon,\zeta}^G$ from \eqref{Eq:LyapunovPerron:2}. Then, using assumptions $(A_n)$, (iii) and (iv), and $(D)$, we have
	\begin{align*}
	&\qquad e^{-\eta t}(\| h^{\epsilon,\zeta}_X(\bar{v}_S(t))-h^{\epsilon,\zeta}_G(\bar{v}_G(t)) \|_{X_m})\\
	&\leq e^{-\eta t}\left\|\epsilon^{-1}\int_{-\infty}^t e^{-\epsilon^{-1}(t-s)A} \operatorname{pr}_{X_F^\zeta}f(\bar{u}(s),\bar{v}_F(s),\bar{v}_S(s))\,ds\right\|_{X_m}\\
	&\ +e^{-\eta t}\left\|\epsilon^{-1}\int_{-\infty}^t e^{-\epsilon^{-1}(t-s)A} \operatorname{pr}_{X_S^\zeta}\big[f(\bar{u}(s),\bar{v}_F(s),\bar{v}_S(s))-f(\bar{u}_G(s),\bar{v}_G(s))\big]\,ds\right\|_{X_m}\\
	&\leq C_{m,n} \zeta^{n-m} e^{-\eta t}\left\|\epsilon^{-1}\int_{-\infty}^t e^{-\epsilon^{-1}(t-s)A} \operatorname{pr}_{X_F^\zeta}f(\bar{u}(s),\bar{v}_F(s),\bar{v}_S(s))\,ds\right\|_{X_n}\\
	&\ +e^{-\eta t}\left\|\epsilon^{-1}\int_{-\infty}^t e^{-\epsilon^{-1}(t-s)A} \operatorname{pr}_{X_S^\zeta}\big[f(\bar{u}(s),\bar{v}_F(s),\bar{v}_S(s))-f(\bar{u}_G(s),\bar{v}_G(s))\big]\,ds\right\|_{X_m}\\
	&\leq L_fC_AC_{m,n}\zeta^{n-m}\int_{-\infty}^t\frac{e^{(\epsilon^{-1} \zeta^{-1}\omega_A-\eta)(t-s)}}{\epsilon^{\gamma_X}(t-s)^{1-\gamma_X}}\,ds \|(\bar{u},\bar{v}_F,\bar{v}_S)\|_{\mathcal C_{\eta,n}}\\
	&\ +L_fC_A\int_{-\infty}^t\frac{e^{(\epsilon^{-1} \omega_A-\eta)(t-s)}}{\epsilon^{\gamma_X}(t-s)^{1-\gamma_X}}\,ds \|(\bar{u}-\bar{u}_G,\bar{v}_F,\bar{v}_S-\bar{v}_G)\|_{\mathcal C_{\eta, m}}.
\end{align*}
It was shown in the proof of \cite[Proposition 5.2]{Hummel_Kuehn_2020} that the mapping
\[
	Y_S^{\zeta}\cap Y_k \to C_{\eta,k}, v_{0,S} \mapsto (\bar{u},\bar{v}_F,\bar{v}_S)
\]
is Lipschitz continuous. Let $L>0$ be the Lipschitz constant. Moreover, \cite[Lemma 2.2]{Hummel_Kuehn_2020} shows that
\[
	\int_{-\infty}^t\frac{e^{(\epsilon^{-1} \omega_A-\eta)(t-s)}}{\epsilon^{\gamma_X}(t-s)^{1-\gamma_X}}\,ds \leq \frac{\Gamma(\gamma_X)}{(\epsilon\eta-\omega_A)^{\gamma_X}}.
\]
Hence, we obtain that
\begin{align}
\begin{aligned}\label{Eq:Ingredient1}
	&e^{-\eta t}(\| h^{\epsilon,\zeta}_X(\bar{v}_S(t))-h^{\epsilon,\zeta}_G(\bar{v}_G(t)) \|_{X_m})	\\
	\leq \frac{LL_f C_AC_{m,n}\Gamma(\gamma_X)}{(\epsilon\eta-\zeta^{-1}\omega_A)^{\gamma_X}}&\zeta^{n-m}\|v_0\|_{Y_n}+\frac{L_f C_A\Gamma(\gamma_X)}{(\epsilon\eta-\omega_A)^{\gamma_X}}\|(\bar{u}-\bar{u}_G,\bar{v}_F,\bar{v}_S-\bar{v}_G)\|_{\mathcal C_{\eta, m}}.
	\end{aligned}
\end{align}
Furthermore, combining \cite[(5-3)]{Hummel_Kuehn_2020} with Assumption $(D)$ yields
\begin{align}
	\begin{aligned}\label{Eq:Ingredient2}
	 e^{-\eta t}\|h^{\epsilon,\zeta}_{Y_F^{\zeta}}(\bar{v}_S(t))\|_{Y_m}&\leq C_{m,n}\zeta^{n-m} e^{-\eta t}\|h^{\epsilon,\zeta}_{Y_F^{\zeta}}(\bar{v}_S(t))\|_{Y_n}\\
	 &\leq  \frac{2^{\delta_Y}C_{m,n}LL_gC_B\Gamma(\delta_Y)}{(N_S^{\zeta}-N_F^{\zeta})^{\delta_Y}}\zeta^{n-m}\|v_0\|_{Y_n}.
	 \end{aligned}
\end{align}
Concerning $\bar{v}_S-\bar{v}_G$, we observe with $(B_n)$, (ii) and (v), that
\begin{align}\begin{aligned}\label{Eq:Ingredient3}
	&e^{-\eta t}\| \bar{v}_S(t) - \bar{v}_G(t) \|_{Y_m}\\ & \quad = e^{-\eta t} \left\| \int_t^0 e^{(t-s)B}\operatorname{pr}_{Y_S^{\zeta}}\big[g(\bar{u}(s),\bar{v}_F(s),\bar{v}_S(s))-g(\bar{u}_G(s),\bar{v}_G(s))\big]\,ds\right\|_{Y_m}\\
& \quad \leq L_gC_B\int_t^0\zeta^{\delta_Y-1}e^{(t-s)(\zeta^{-1}\omega_A-\eta)}\,ds\|(\bar{u}-\bar{u}_G,\bar{v}_F,\bar{v}_S-\bar{v}_G)\|_{\mathcal C_{\eta,m}}\\
	& \quad \leq \frac{L_gC_B\zeta^{\delta_Y-1}}{N_S^{\zeta}-N_F^{\zeta}}\|(\bar{u}-\bar{u}_G,\bar{v}_F,\bar{v}_S-\bar{v}_G)\|_{\mathcal C_{\eta,m}}.
\end{aligned}\end{align}
Summing up \eqref{Eq:Ingredient1}, \eqref{Eq:Ingredient2} and \eqref{Eq:Ingredient3} yields
\begin{align*}
	&\quad e^{-\eta t}\big[\| h^{\epsilon,\zeta}_X(\bar{v}_S(t))-h^{\epsilon,\zeta}_G(\bar{v}_G(t)) \|_{X_m}+\|h^{\epsilon,\zeta}_{Y_F^{\zeta}}(\bar{v}_S(t))\|_{Y_m}+\| \bar{v}_S(t) - \bar{v}_G(t) \|_{Y_m}\big]\\
&\leq C_{m,n}\left(\frac{2^{\delta_Y}LL_gC_B\Gamma(\delta_Y)}{(N_S^{\zeta}-N_F^{\zeta})^{\delta_Y}}+\frac{LL_f C_A\Gamma(\gamma_X)}{(\epsilon\eta-\omega_A)^{\gamma_X}}\right)\zeta^{n-m}\|v_0\|_{Y_n}\\
&+ \left( \frac{L_f C_A\Gamma(\gamma_X)}{(\epsilon\eta-\omega_A)^{\gamma_X}}+ \frac{L_gC_B\zeta^{\delta_Y-1}}{N_S^{\zeta}-N_F^{\zeta}}\right)\|(h^{\epsilon,\zeta}_X(\bar{v}_S)-h^{\epsilon,\zeta}_G(\bar{v}_G),h^{\epsilon,\zeta}_{Y_F^{\zeta}}(\bar{v}_S(t)),\bar{v}_S-\bar{v}_G)\|_{\mathcal C_{\eta,m}}.
\end{align*}
Therefore, if we write
	\[
		\tilde{L}:=\left( \frac{L_f C_A\Gamma(\gamma_X)}{(\omega_A-\epsilon\eta)^{\gamma_X}}+ \frac{L_gC_B\zeta^{\delta_Y-1}}{N_S^{\zeta}-N_F^{\zeta}}\right),
	\]
	which is strictly smaller than $1$ by Assumption $(B_m)$, we obtain
	\begin{align*}
		&\|(h^{\epsilon,\zeta}_X(\bar{v}_S)-h^{\epsilon,\zeta}_G(\bar{v}_G),h^{\epsilon,\zeta}_{Y_F^{\zeta}}(\bar{v}_S),\bar{v}_S-\bar{v}_G)\|_{\mathcal C_{\eta,m}} \\ &\leq\frac{C_{m,n}}{1-\tilde{L}}\left(\frac{2^{\delta_Y}LL_gC_B\Gamma(\delta_Y)}{(N_S^{\zeta}-N_F^{\zeta})^{\delta_Y}}+\frac{LL_f C_A\Gamma(\gamma_X)}{(\zeta^{-1}\omega_A-\epsilon\eta)^{\gamma_X}}\right)\zeta^{n-m}\|v_0\|_{Y_n}.
	\end{align*}
	The fact that
	\begin{align*}
	&\|h^{\epsilon,\zeta}_X(v_{0,S})-h^{\epsilon,\zeta}_G(v_{0,S})\|_{X_m} + \|h^{\epsilon,\zeta}_{Y_F^{\zeta}}(v_{0,S})\|_{Y_m} \\
& \qquad	\leq  \|(h^{\epsilon,\zeta}_X(\bar{v}_S)-h^{\epsilon,\zeta}_G(\bar{v}_G),h^{\epsilon,\zeta}_{Y_F^{\zeta}}(\bar{v}_S),\bar{v}_S-\bar{v}_G)\|_{C_{\eta,m}}
	\end{align*}
	finally yields the assertion.
\end{proof}

\section{Case study of an explicit reaction-diffusion problem}
\label{sec:case_study}
As discussed in Section~\ref{Sec:Galerkin}, in certain situations of interest the spaces $X_S^{\zeta}$ and $Y_S^{\zeta}$ are $N_A$-dimensional and $N_B$-dimensional with $N_A$ and $N_B$ being the number of eigenvalues including multiplicities of $A$ and $B$, respectively, in $\{z\in\C: \zeta^{-1}\omega_A+N_S^{\zeta}\leq\Re(z)\}$. In a Galerkin approach, one usually studies the limit $N_A,N_B\to\infty$ which corresponds to $\zeta\to0$. However, when we fix $\epsilon > 0$, the condition $ (c \, \omega_A/\omega_f)\zeta>\epsilon$ for some $c\in(0,1)$, as posed in Theorem~\ref{Thm:Main}, will not be satisfied in the limit $\zeta\to0$. In such a situation, the existence of slow manifolds in the sense of \cite{Hummel_Kuehn_2020} is unclear. Moreover, in the limit $\zeta\to0$ the interpretation of \eqref{Eq:Fast-Slow_Galerkin} as a classical finite-dimensional fast-slow system will be lost, since the lower order modes in the fast variable might evolve at a slower time scale than the higher order modes in the slow variable. Therefore, one may ask whether slow manifolds for \eqref{Eq:Fast-Slow_Galerkin} still exist in a suitable sense if $\zeta\in(0,\epsilon \, \omega_f/(c \,\omega_A ) )$.

In a general setting, it is not clear whether this holds true or not. However, in certain situations it is possible to explicitly derive invariant manifolds for \eqref{Eq:Fast-Slow_Galerkin} which resemble slow manifolds from the classical finite-dimensional theory. Using such a computation, we discuss the intricacies of the limit $\zeta\to0$  at the hand of an example, also providing an explicit estimate of the form~\eqref{Eq:Comparison_slow_manifolds}.

\begin{example} \label{Ex:explicit_calculations}
	(i) \textbf{Explicit computation of slow manifolds}: consider the following fast-slow system
\begin{align}\label{Eq:Fast_Slow:Example}
	\begin{aligned}
	\epsilon\partial_t u^{\epsilon}&=\Delta u^{\epsilon} - u^{\epsilon} + (v^{\epsilon})^2,\\
	\partial_t v^{\epsilon}&=\Delta v^{\epsilon}-v^{\epsilon},
	\end{aligned}
\end{align}
on the torus $\mathbb{T}$. A natural approach for a Galerkin approximation is to truncate to a certain number of Fourier modes. Writing 
	\[
		u_k^{\epsilon}(t):=\int_{\mathbb{T}} u^{\epsilon}(t,x) e^{-2\pi i k x}\,dx,\quad v_k^{\epsilon}(t):=\int_{\mathbb{T}} v^{\epsilon}(t,x) e^{-2\pi i k x}\,\txtd x\quad(k\in\Z),
	\]
	we can expand 
\[
	u^{\epsilon}(t,x)=\sum_{k\in\Z} u_k^{\epsilon}(t) e^{2\pi i k x},\quad v^{\epsilon}(t,x)=\sum_{k\in\Z} v_k^{\epsilon}(t) e^{2\pi i k x}.
\]
Applying $\langle \cdot, e^{2\pi i k x} \rangle_{L_2(\mathbb{T})}$ to both sides of \eqref{Eq:Fast_Slow:Example} yields
\begin{align}\label{Eq:Fast_Slow_Fourier_Series_Slow}
	\begin{aligned}
	\epsilon\partial_t u^{\epsilon}_k&=-(1+4\pi^2k^2)u^{\epsilon}_k + \sum_{j,l\in\Z, j+l=k} v_j^{\epsilon}v_l^{\epsilon},\\
	\partial_t v^{\epsilon}_k&=-(1+4\pi^2k^2)v^{\epsilon}_k
	\end{aligned}
\end{align}
for all $k\in\Z$. Truncating at a certain $k_0\in\N$, we obtain
\begin{align}\label{Eq:Fast_Slow_Fourier_Series_Slow_Truncation}
	\begin{aligned}
	\epsilon\partial_t u^{\epsilon}_k&=-(1+4\pi^2k^2)u^{\epsilon}_k + \sum_{j,l\in\Z,\;|j|,|l|\leq k_0,\atop j+l=k} v_j^{\epsilon}v_l^{\epsilon},\\
	\partial_t v^{\epsilon}_k&=-(1+4\pi^2k^2)v^{\epsilon}_k,
	\end{aligned}
\end{align}
which is a finite-dimensional fast-slow ODE for sufficiently small $\epsilon>0$.
We can directly solve the slow equation by
\begin{align*}
	v_k^{\epsilon}(t)=e^{-(1+4\pi^2k^2)t}v_k^{\epsilon}(0),
\end{align*}
and, if $[\epsilon^{-1}(1+4\pi^2k^2)-2-4\pi^2(j^2+l^2)]\neq0$, the fast equation is solved by
\begin{align*}
	&u_k^{\epsilon}(t) - e^{-\epsilon^{-1}(1+4\pi^2k^2)t}u_k^{\epsilon}(0)\\
	&=\sum_{j,l\in\Z,\;|j|,|l|\leq k_0,\atop j+l=k} \epsilon^{-1}\int_0^t e^{-\epsilon^{-1}(1+4\pi^2k^2)(t-s)}v_j^{\epsilon}(s)v_l^{\epsilon}(s)\, \txtd s\\
	&=\sum_{j,l\in\Z,\;|j|,|l|\leq k_0,\atop j+l=k} \epsilon^{-1}e^{-\epsilon^{-1}(1+4\pi^2k^2)t}\int_0^t e^{[\epsilon^{-1}(1+4\pi^2k^2)-2-4\pi^2(j^2+l^2)]s}v_j^{\epsilon}(0)v_l^{\epsilon}(0)\, \txtd s\\
	&=\sum_{j,l\in\Z,\;|j|,|l|\leq k_0,\atop j+l=k}\epsilon^{-1}\frac{e^{-2-4\pi^2(j^2+l^2)]t}-e^{-\epsilon^{-1}(1+4\pi^2k^2)t}}{\epsilon^{-1}(1+4\pi^2k^2)-2-4\pi^2(j^2+l^2)}v_j^{\epsilon}(0)v_l^{\epsilon}(0).
\end{align*}
The essential property of a slow manifold is to eliminate the fast dynamics. In fact, cancelling out the terms with the $\epsilon^{-1}$ in the exponent, we obtain 
\begin{align}\label{Eq:Slow_Manifold}
	u_k^{\epsilon}(0)=\sum_{j,l\in\Z,\;|j|,|l|\leq k_0,\atop j+l=k}\frac{v_j^{\epsilon}(0)v_l^{\epsilon}(0)}{1+4\pi^2k^2-\epsilon[2+4\pi^2(j^2+l^2)]},
\end{align}
which could be seen as a formula for the slow manifold. The critical manifold in turn would be given by
\begin{align}\label{Eq:Critical_Manifold}
	u_k^{0}(0)=\sum_{j,l\in\Z,\;|j|,|l|\leq k_0,\atop j+l=k}\frac{v_j^{0}(0)v_l^{0}(0)}{1+4\pi^2k^2}.
\end{align}
Let $\mathcal{M}_{k_0}$ be the set of all $\epsilon\in(0,1)$ such that there are $(j,k)\in\Z^2$ with $ \max\{|j|,|l|\}\leq k_0$ and $k\in\Z$, $|k|\leq k_0$ with 
\begin{align} \label{eq:epsilon_k0}
	\epsilon^{-1}(1+4\pi^2k^2)-2-4\pi^2(j^2+l^2)=0,
\end{align}
i.e., $\mathcal{M}_{k_0}$ contains all $\epsilon\in(0,1)$ for which there may be singularities in \eqref{Eq:Slow_Manifold}. This set is special since the above procedure of cancelling out the terms with an $\epsilon^{-1}$ in the exponent is not possible for such $\epsilon$; note that a similar situation occurs  in many dynamical systems in the context of resonances and the small divisor problem~\cite{Yoccoz}. 
Although the existence of invariant manifolds is not clear for $\epsilon\in \mathcal{M}_{k_0}$, we observe that $\mathcal{M}_{k_0}$ is finite if $k_0\in\N$ and countable with an accumulation point at $0$ if $k_0=\infty$. This shows two things: firstly, for all but countably many $\epsilon\in(0,1)$ there exists an invariant manifold as a graph over the whole slow variable space for \eqref{Eq:Fast_Slow:Example}. Secondly, it seems like there is no $\epsilon_0$ such that such an invariant manifold exists for all $\epsilon\in(0,\epsilon_0]$. Instead, 
one has to restrict to a subset of the slow variable space,
as also suggested by the direct approach. In this example, one has to impose 
\[
	\sum_{(j,k)\in \mathcal{L}_{k_0,k }} v_j^{\epsilon}(0)v_l^{\epsilon}(0)=0,
\]
where $\mathcal{L}_{k_0,k }$ denotes the set of all pairs $(j,l)\in\Z^2$ with $j+l=k$ and $|j|,|l|\leq k_0$ such that the denominator in \eqref{Eq:Slow_Manifold} is equal $0$.

Even though an invariant manifold for \eqref{Eq:Fast_Slow:Example} exists for all but countably many $\epsilon\in(0,1)$, these manifolds can be far away from the critical manifold. In fact, this distance tends to $\infty$ as $\dist(\epsilon, \mathcal M_{k_0}) \to 0$, as can be seen directly from \eqref{Eq:Slow_Manifold} and \eqref{Eq:Critical_Manifold}.
However, it is easy to see from equation~\eqref{eq:epsilon_k0} that there are no singularities in \eqref{Eq:Fast_Slow:Example} if $\epsilon < 1/(2+8\pi^2k_0^2)$. In this case, the slow manifold from the Galerkin approximation is close to the slow manifold obtained by the direct approach. This can be checked by computing the slow manifold for \eqref{Eq:Fast_Slow:Example} from the Fourier coefficients as above. In our abstract framework, we obtain the following precise estimate.

\medskip

(ii) \textbf{Exemplification of abstract framework}: one can choose $X=L_2(\mathbb{T})$ and $Y=H^2(\mathbb{T})$ as underlying spaces and $A=\Delta-1$ and $B=\Delta-1$ on the domains $D(A)=H^2(\mathbb{T})$ and $D(B)=H^4(\mathbb{T})$, respectively. Then, we have $X_{\alpha}=H^{2\alpha}(\mathbb{T})$ and $Y_{\alpha}=H^{2+2\alpha}(\mathbb{T})$. As nonlinearities, we choose $f(x,y):=y^2$ and $g(x,y)=0$ with $\gamma_X=\delta_X=\delta_Y=1$. If $n\geq1$, then 
\[
	f\colon X_{n}\times Y_{n-1}\to X_{n},\,(x,y)\mapsto y^2
\]
is a well-defined and smooth nonlinearity, since $H^{2+2\alpha}(\mathbb{T})$ is a Banach algebra for $\alpha>-\frac{3}{4}$. However, the bounds on the derivatives of $f$ from Assumption $(A_n)$ are only satisfied locally and the Lipschitz constant only gets small in a neighborhood around $0$ in $Y_{n}$. Formally, one would have to use cutoff techniques as for example in \cite[Section 6]{Hummel_Kuehn_2020} in order to apply our methods. But since globabl stability issues are not our primary concern, we omit the details here. Instead, we just keep in mind that we have to restrict to a certain neighborhood around $0$ in $Y_{n}$ so that $L_f$ is small enough for \eqref{Eq:Estimate_for_FixedPoint} to hold.

For Assumption $(B_n)$ we need to introduce a splitting of the slow variable space. Let $\omega_A\in(-1,0)$ be close to $-1$. For $k_0\in\N_0$ with $$-4\pi^2(|k_0|+2)^2<\zeta^{-1}\omega_A+1\leq -4\pi^2(|k_0|+1)^2,$$
 we take 
\begin{align*}
	X_S^{\zeta}=Y_S^{\zeta}&:=\operatorname{span}\{[x\mapsto e^{i2\pi k x}]:k\in\Z,\;|k|\leq|k_0|\},\\
	Y_F^{\zeta}&:=\operatorname{cl}_{H^2(\mathbb{T})}\big(\operatorname{span}\{[x\mapsto e^{i2\pi k x}]:k\in\Z,\;|k|\geq|k_0|+1\}\big), \\
	X_F^{\zeta}&:=\operatorname{cl}_{L_2(\mathbb{T})}\big(\operatorname{span}\{[x\mapsto e^{i2\pi k x}]:k\in\Z,\;|k|\geq|k_0|+1\}\big), 
\end{align*}
where $\operatorname{cl}_{T}M$ denotes the closure of a set $A\subset T$ in a topological space $T$. Note that the projection to $X_S^{\zeta}$ and $Y_S^{\zeta}$ coinides with the projection to the first $k_0$ Fourier modes. Thus, our abstract Galerkin equation \eqref{Eq:Fast-Slow_Galerkin} is consistent with the explicit example \eqref{Eq:Fast_Slow_Fourier_Series_Slow_Truncation}. Now it is straightforward to check that the assumptions $(B_n)$, $(C_n)$ and $(D)$ are satisfied. Nevertheless, let us specify the choice of $N_F^{\zeta}$ and $N_S^{\zeta}$. Note that we have
\[
	e^{tB}f=e^{t(\Delta-1)}f=\left[x\mapsto \sum_{k\in\Z} e^{-(4\pi^2 k^2+1)t}\hat{f}(k)e^{i2\pi k x}\right].
\]
Thus, for $y_S\in Y_S^{\zeta}$ and $t\geq0$, Plancherel's Theorem gives
\[
	\|e^{-tB}y_S\|_{H^{2+2n}(\mathbb{T})}\leq e^{(4\pi^2|k_0|^2+1)t}\|y_S\|_{H^{2+2n}(\mathbb{T})},
\]
so that we may take
\[
	N_S^{\zeta}:=-\zeta^{-1}\omega_A-4\pi^2|k_0|^2-1
\]
Since $-4\pi^2(|k_0|+2)^2<\zeta^{-1}\omega_A+1\leq -4\pi^2(|k_0|+1)^2$, it follows that $N_S^{\zeta}>0$. Similarly, we can take 
\[
	N_F^{\zeta}:=-\zeta^{-1}\omega_A-4\pi^2(|k_0|+1)^2-1.
\]
With these choices, we observe that formula \eqref{Eq:Slow_Manifold} defines the slow manifold which one also obtains from a Lyapunov-Perron approach. Indeed, the solution of equation \eqref{Eq:Fast_Slow_Fourier_Series_Slow_Truncation}, with initial conditions given by \eqref{Eq:Slow_Manifold}, reads
\begin{align*}
		u_k^{\epsilon}(t)&=\sum_{j,l\in\Z,\;|j|,|l|\leq k_0,\atop j+l=k}e^{-[2-4\pi^2(j^2+l^2)]t}\frac{v_j^{\epsilon}(0)v_l^{\epsilon}(0)}{1+4\pi^2k^2-\epsilon[2+4\pi^2(j^2+l^2)]},\\
		v_k^{\epsilon}(t)&=e^{-(1+4\pi^2k^2)t}v_k^{\epsilon}(0),
\end{align*}
also for $t\in(-\infty,0]$. Moreover, this solution is an element of $\mathcal C_{\eta,n}^G$ and hence, \eqref{Eq:Slow_Manifold} defines the slow manifold given as the graph of $h_G^{\epsilon,\zeta}$ from the abstract setting. In particular, Theorem~\ref{Thm:Main} shows that the distance of the Galerkin slow manifold we computed for \eqref{Eq:Fast_Slow_Fourier_Series_Slow_Truncation} to the actual slow manifold for \eqref{Eq:Fast_Slow:Example} is small if $\zeta,\epsilon>0$ with $c\frac{\omega_A}{\omega_f}\zeta>\epsilon$ are small enough. More precisely, if we fix $m,n\in\N$, $m\leq n$ and $c\in (0,1)$, then Theorem~\ref{Thm:Main} tells us that there is a constant $C>0$ such that for all $\epsilon>0$ small enough and all $k_0\in\N$ such that $k_0<\sqrt{c\omega_f/(4\pi^2\epsilon)}-2$ it holds that
\[
	\| h_X^{\epsilon,\zeta}(v_{0,S})-h_G^{\epsilon,\zeta}(v_{0,S}) \|_{H^{2m}(\mathbb{T})} < C k_0^{-2(n-m)-1}\|v_{0,S}\|_{H^{2n+2}(\mathbb{T})} .
\]
Here, $h_X^{\epsilon,\zeta}$ denotes the mapping describing the slow manifold for \eqref{Eq:Fast_Slow:Example} from the direct approach and $h_G^{\epsilon,\zeta}$ denotes the slow manifold from the Galerkin approach defined by \eqref{Eq:Slow_Manifold}. If $k_0$ is chosen close enough to $\sqrt{c\omega_f/(4\pi^2\epsilon)}-2$, then we also obtain the estimate
 \[
	\| h_X^{\epsilon,\zeta}(v_{0,S})-h_G^{\epsilon,\zeta}(v_{0,S}) \|_{H^{2m}(\mathbb{T})} \lesssim \epsilon^{n-m+\frac{1}{2}}\|v_{0,S}\|_{H^{2n+2}(\mathbb{T})} .
\] 
In particular, the last estimate provides an illustration of the relevance of our main result: 
in situations where a Galerkin approximation may be the procedure of choice due to the need of using ODE techniques or for numerical reasons, we know that for sufficiently small $\epsilon$ and suitably chosen $k_0$ the finite-dimensional Galerkin manifolds are good approximations of the invariant slow manifolds for the PDE, if the appropriate norms are taken.
\end{example}

\appendix

\section{Interpolation-Extrapolation Scales} \label{sec:int_ext_scales}

We briefly recall the notion of interpolation-extrapolation scales and related results. As a general reference, see \cite[Chapter V]{Amann_1995}.
Let $T\colon X\supset D(T)\to X$ be a densely defined closed linear operator on a Banach space $X$ with $0\in\rho(T)$, where $\rho(T)$ denotes its resolvent set. Moreover, for $\theta\in(0,1)$ let $(\cdot,\cdot)_{\theta}$ be an exact admissible interpolation functor, i.e.~an exact interpolation functor such that $X_1$ is dense in $(X_0,X_1)_{\theta}$ whenever $X_1\stackrel{d}{\hookrightarrow} X_0$ (i.e.~the injection is continuous with dense range). We define a family of Banach spaces $(X_{\alpha})_{\alpha\in[-1,\infty)}$ and a family of operators $(T_{\alpha})_{\alpha\in[-1,\infty)}\in\mathcal{B}(X_{\alpha+1},X_{\alpha})$ as follows:
 \begin{itemize}
  \item For $n\in\N_0$ we choose $X_n:=D(T^n)$ endowed with $\|x\|_{X_n}:=\|T^n x\|_X$ $(x\in D(T^n))$. In particular, $X_0=D(T^0)=D(\operatorname{id}_X)=X$. Moreover, $T_n:=T\vert_{X_{n+1}}$.
  \item $X_{-1}$ is defined as the completion of $X=X_0$ with respect to the norm $\|x\|_{X_{-1}}=\| T^{-1} x \|_{X_0}$. The operator $T_{0}=T$ is then closable on $X_{-1}$ and $T_{-1}$ is defined to be the closure. One can also define  $(X_{-n},T_{-n})$ for $n\in\N$ by iteration, but we do not go beyond $n=-1$ in this paper.
  \item For $n\in\N_0\cup\{-1\}$, $\theta\in(0,1)$ and $\alpha=n+\theta$ we define $X_{\alpha}:=(X_{n},X_{n+1})_{\theta}$ and $T_{\alpha}=T_n\vert_{D(T_{\alpha})}$ where
  \[
   D(T_{\alpha})=\{x\in X_{n+1}: T_n x\in X_{\alpha}\}.
  \]
 \end{itemize}
The family $(X_{\alpha},T_{\alpha})_{\alpha\in[-1,\infty)}$ is a densely injected Banach scale in the sense that
$
 X_{\alpha}\stackrel{d}{\hookrightarrow} X_{\beta}
$
whenever $\alpha\geq\beta$,  and
$
 T_{\alpha}\colon X_{\alpha+1}\to X_{\alpha}
$
is an isomorphism for all $\alpha\in\R$. Moreover $T_{\alpha}\colon X_{\alpha}\supset X_{\alpha+1}\to X_{\alpha}$ is a densely defined closed linear operator with $0\in\rho(T_{\alpha})$ for all $\alpha\in\R$. The family $(X_{\alpha},T_{\alpha})_{\alpha\in\R}$ is an interpolation-extrapolation scale.
%

\end{document}